\documentclass[reqno, 10pt]{amsart}
\usepackage{amssymb}
\usepackage{amsmath}
\usepackage[mathscr]{euscript}
\usepackage{hyperref}
\usepackage{graphicx}
\usepackage[normalem]{ulem}

\usepackage{amssymb}
\usepackage{hyperref}
\RequirePackage[dvipsnames]{xcolor} 
\definecolor{halfgray}{gray}{0.55} 
\definecolor{webgreen}{rgb}{0,0.5,0}
\definecolor{webbrown}{rgb}{.6,0,0} \hypersetup{%
  colorlinks=true, linktocpage=true, pdfstartpage=3,
  pdfstartview=FitV,%
  breaklinks=true, pdfpagemode=UseNone, pageanchor=true,
  pdfpagemode=UseOutlines,%
  plainpages=false, bookmarksnumbered, bookmarksopen=true,
  bookmarksopenlevel=1,%
  hypertexnames=true,
  pdfhighlight=/O,
  urlcolor=webbrown, linkcolor=RoyalBlue,
  citecolor=webgreen, 
  pdftitle={},%
  pdfauthor={},%
  pdfsubject={2000 MAthematical Subject Classification: Primary:},%
  pdfkeywords={},%
  pdfcreator={pdfLaTeX},%
  pdfproducer={LaTeX with hyperref}%
}

\usepackage{enumitem}

\newtheorem{theorem}{Theorem}[section]
\newtheorem{lemma}[theorem]{Lemma}
\newtheorem{corollary}[theorem]{Corollary}
\newtheorem{proposition}[theorem]{Proposition}

\theoremstyle{definition}

\def\R{\mathbb{R}}
\def\N{\mathbb{N}}
\def\Z{\mathbb{Z}}
\def\Reg{\mathcal{R}^{\mu}}
\def\Lam{\Lambda^{\mu}}
\newcommand{\norm}[1]{{\left\lVert \, #1 \, \right\rVert}}

\def\O{\mathcal{O}}
\def\cO{\mathfrak{O}}
\def\Id{\text{Id}}
\def\GL{GL(d,\mathbb{R})}
\def\supp{\text{supp}}
\def\cHCT{\mathcal{H}(C,\tau)}
\def\cN{\mathcal{N}}
\def\cDN{\mathcal{D}_A(N,\theta)}

\def\Ws{W^s_{loc}}
\def\Wu{W^u_{loc}}

\def\d{\operatorname{dist}}

\begin{document}

\title
[A Liv\v{s}ic theorem for matrix cocycles]
{A Liv\v{s}ic theorem for matrix cocycles over non-uniformly hyperbolic systems}

\author[Lucas Backes]{Lucas Backes}
\address{Departamento de Matem\'atica, Universidade Federal do Rio Grande do Sul, Av. Bento Gon\c{c}alves 9500, CEP 91509-900, Porto Alegre, RS, Brazil.}
\email{lhbackes@impa.br }

\author{Mauricio Poletti}
\address{LAGA -- Universit\'e Paris 13, 99 Av. Jean-Baptiste Cl\'ement, 93430 Villetaneus, France.}
\email{mpoletti@impa.br}

\date{\today}


\begin{abstract}
We prove a Liv\v{s}ic-type theorem for H\"older continuous and matrix-valued cocycles over non-uniformly hyperbolic systems. More precisely, we prove that whenever $(f,\mu)$ is a non-uniformly hyperbolic system and $A:M \to GL(d,\mathbb{R}) $ is an $\alpha$-H\"{o}lder continuous map satisfying $ A(f^{n-1}(p))\ldots A(p)=\text{Id}$ for every  $p\in \text{Fix}(f^n)$ and $n\in \N$, there exists a measurable map $P:M\to GL(d,\mathbb{R})$ satisfying $A(x)=P(f(x))P(x)^{-1}$ for $\mu$-almost every $x\in M$. Moreover, we prove that whenever the measure $\mu$ has local product structure the transfer map $P$ is $\alpha$-H\"{o}lder continuous in sets with arbitrary large measure.

\end{abstract}

\maketitle

\section{Introduction}

Consider an invertible measure preserving transformation $f:(M,\mu)\to (M,\mu)$ of a standard probability space. Given a measurable map $A:M\to G$, where $G$ denotes a topological group, we are interested in determining whether there exists a map $P:M\to G$ satisfying
\begin{equation}\label{eq: transf map}
A(x)=P(f(x))P(x)^{-1} \text{ for $\mu$-almost every } x\in M.
\end{equation}
When that is the case, $A$ is said to be a \emph{coboundary}. Moreover, whenever \eqref{eq: transf map} does admit a solution, usually called \emph{transfer map}, we are interested in understanding its regularity properties. 

These problems go back to the seminal papers \cite{Liv71, Liv72} of Liv\v{s}ic where he have proved that, whenever $f$ is hyperbolic, $(G,\cdot)$ is an Abelian group and $A$ is H\"{o}lder continuous, $A$ is a coboundary with a H\"{o}lder transfer map if and only if 
\begin{equation}\label{eq: POOC intro}
A(f^{n-1}(p))\cdot \ldots \cdot A(f(p))A(p)=\Id \text{ for every } p\in \text{Fix}(f^n) \text{ and } n\in \N
\end{equation}
where $\Id$ denotes the identity element of $G$. Since then, due to its far reaching applications, this result has been generalized in many different directions. For instance, by considering 

\emph{$\circ$ More general groups.} Suppose $f$ is hyperbolic. It is easy to see that condition \eqref{eq: POOC intro} is necessary for the existence of continous solutions of \eqref{eq: transf map} for any group $G$. Thus, the main challenge is to know whether this condition is also sufficient when considering cocycles taking values in more general groups. Affirmative answers were given in several contexts. Initially, the main technique used to deal with this problem was to ``control distortions" and then proceed as in the Abelian case. In order to do so, \emph{localization hypotheses} were used (see for instance \cite{Liv72, PW01,  dlLW10, KN11} and references therein). The first complete solution to this problem, with no localization hypotheses, was given by Kalinin \cite{Kal11} in the case when $G=\GL$ and by \cite{KP16, AKL} in the case when $G=\text{Diff}^{1+\epsilon}(N)$ is the diffeomorphism group of any closed manifold $N$. The more general problem of \emph{cohomology} was also consider by many authors (see for instance \cite{Par99, Sch99, Bac15, Sa15, BK16} and references therein);

\emph{$\circ$ More general dynamics.} Suppose $(G,\cdot)=(\mathbb{R},+)$. The main difficulty when trying to consider more general base dynamics, like \emph{partially hyperbolic systems}, for instance, is that, unlike transitive hyperbolic diffeomorphisms where the set of periodic points is dense, a transitive partially hyperbolic diffeomorphism might have no periodic orbits at all (for instance, one can take the time-$t$ map of a transitive Anosov flow, for an appropriate choice of t). Hence condition \eqref{eq: POOC intro} can be empty. Thus, in order to solve \eqref{eq: transf map} new obstructions are needed. In \cite{KK96}, Katok and Kononenko presented a new set of obstructions suitable to this context which were also used by Wilkinson in \cite{Wil13}. Another natural class of systems to consider as base dynamics is that of \emph{non-uniformly hyperbolic systems}. In this context we still have plenty of periodic points, at least close to the support of the hyperbolic measure, and thus condition \eqref{eq: POOC intro} still can be used. To the best of our knowledge, the most general result gives us only a measurable solution of \eqref{eq: transf map} whenever \eqref{eq: POOC intro} is satisfied (see Theorem S.4.17 of \cite{KH95} or Theorem 15.3.1 of \cite{BP07} and the end of the next paragraph). 

Another important line of research in this area is the study of regularity properties of solutions to \eqref{eq: transf map}. In general, whenever the base dynamics is hyperbolic one can recover the regularity of the data. To ilustrate our claim, whenever $f$ is a $C^r$ hyperbolic diffeomorphism and $A\in C^r(M,\mathbb{R})$ for any non-integer $r>1$, every continuous solution to \eqref{eq: POOC intro} is $C^r$ (see for instance \cite{dlLMM86} and references therein). Similary, if $A$ is H\"older continuous then any measurable solution $P$ has a version that is also H\"older continuous \cite{Liv72}. For the case of more general groups, see \cite{NT96,NT98}. In the case when the base dynamics is partially hyperbolic, the most general result is due to Wilkinson \cite{Wil13} and says that if $A\in C^k(M,\mathbb{R})$ for $k\geq 2$ and $f$ is a partially hyperbolic, accessible and strongly $r$-bunched diffeomorphism for some $r < k-1$ or $r=1$ then any continuous solution to \eqref{eq: transf map} is $C^r$. For non-uniformly hyperbolic systems $(f,\mu)$, however, it is not always possible to get good regularity for $P$ in the ``whole  space". For instance, in \cite{Pol05}, Pollicot presented a Lipschitz map $A:M\to \mathbb{R}$ admiting a measurable solution to \eqref{eq: transf map} that does not have a H\"older continuos version. However, he was able to prove that, for $\mu=\text{Lebesgue}$ and any H\"{o}lder continuos map $A:M\to \mathbb{R}$, any measurable solution to \eqref{eq: transf map} is H\"{o}lder continuous on sets of arbitrary large measure.

In the present paper we consider both problems, that is, the existence of solutions to \eqref{eq: transf map} and the regularity properties of the transfer map, whenever it exists, in the case when the base dynamics is non-uniformly hyperbolic and the cocycle takes values in the group $G=\GL$. In fact, we extend the ``existence theorem" (Theorem S.4.17 of \cite{KH95} and Theorem 15.3.1 of \cite{BP07}) and Pollicot's regularity result \cite{Pol05} to this context. Our approach is inspired on that of \cite{Kal11} and \cite{Pol05}.

\subsection{Main results} The main results of this work are the following (see Section \ref{sec: prelim} for precise definitions):

\begin{theorem}\label{theo: 1}
Let $f:M\to M$ be a $C^{1+\epsilon}$ diffeomorphism, $\mu$ a hyperbolic ergodic $f$-invariant measure and $A:M \to \GL $ an $\alpha$-H\"{o}lder continuous map. Assume that 
\begin{equation}\label{eq: POOC}
A^n(p)=\Id \text{ for every } p\in \text{Fix}(f^n) \text{ and } n\in \N.
\end{equation}
Then, there exists a measurable map $P:M\to GL(d,\R)$ so that
$$A(x)=P(f(x))P(x)^{-1} \text{ for $\mu$-almost every } x\in M.$$
\end{theorem}

Moreover, we show that whenever the measure $\mu$ has \emph{local product structure}, any transfer map $P$ as in the previous theorem is H\"{o}lder continuous in sets of arbitrary large measure. More precisely,

\begin{theorem}\label{theo: regularity}
Let $f:M\to M$ be a $C^{1+\epsilon}$ diffeomorphism, $\mu$ a hyperbolic ergodic $f$-invariant measure with local product structure and $A:M \to \GL $ an $\alpha$-H\"{o}lder continuous map. Suppose there exists a measurable map $P:M\to GL(d,\R)$ satisfying
\begin{equation}\label{eq: trasnf map theo}
A(x)=P(f(x))P(x)^{-1} \text{ for $\mu$-almost every } x\in M.
\end{equation}
Then, for every $\varepsilon>0$ there exists a set $\Delta_\varepsilon\subset M$ with $\mu(\Delta_\varepsilon)>1-\varepsilon$ so that the map $P$ restricted to $\Delta_\varepsilon$ is $\alpha$-H\"{o}lder continuous.
\end{theorem}

The precise definition of local product structure is given in Section~\ref{sec: more non-unif hyp}. For now, we would just like to stress that many important classes of measures satisfy this property. For instance, the Lebesgue measure (whenever it is hyperbolic) and equilibrium states associated to H\"{o}lder potentials and Axiom A diffeomorphisms have local product structure (see \cite{Lep00}).

Regarding the organization of the paper, Sections~\ref{sec: prelim} and \ref{sec:theo1} contain the proof of Theorem~\ref{theo: 1} while in Section~\ref{sec: holder reg} Theorem~\ref{theo: regularity} is proved. The paper is written in such a way that the reader interested in only one of the results can skip the other part.

\section{Preliminaries and Notations}\label{sec: prelim}

Let $M$ be a closed smooth manifold, $f:M\to M$ a $C^{1+\epsilon}$ diffeomorphism and $\mu$ an ergodic $f$-invariant measure. 

\subsection{Linear cocycles and Lyapunov exponents} \label{sec: Lyap exponents}
Given an integer $d\ge 1$, the \emph{linear cocycle generated by a matrix-valued map $A:M\rightarrow \GL$ over $f$} is the (invertible) map $F_A:M\times \R ^d\rightarrow M\times \R ^d$ defined by
\begin{equation}
\nonumber F_A\left(x,v\right)=\left(f(x),A(x)v\right).
\end{equation}
Its iterates are ${F}^n_A\left(x,v\right)=\left(f^n(x),A^n(x)v\right)$, where
\begin{equation*}\label{def:cocycles}
A^n(x)=
\left\{
	\begin{array}{ll}
		A(f^{n-1}(x))\cdots A(f(x))A(x)  & \mbox{if } n>0 \\
		{\rm Id} & \mbox{if } n=0 \\
		A(f^{n}(x))^{-1}\cdots A(f^{-1}(x))^{-1}& \mbox{if } n<0 .\\
	\end{array}
\right.
\end{equation*}
Sometimes we denote this cocycle by $(f,A)$ or simply by $A$, when there is no risk of ambiguity. A natural example of linear cocycle is given by the \textit{derivative cocycle}: the cocycle generated by $A(x)=Df(x)$ over $f$.

When $\log\norm{A}$ and $\log\norm{ A^{-1}}$ are both integrable, 
a famous theorem of Oseledets~\cite{Ose68} (see also \cite{LLE}) guarantees the existence of a full $\mu$-measure set $\Reg\subset M$, whose points are called \emph{$\mu$-regular}, such that for every $x\in \Reg$ there exist real numbers $\lambda_1 \left(A, x\right)>\cdots >\lambda_l\left(A, x\right)$ and a direct sum decomposition $\R^d=E^{1,A}_{x}\oplus \cdots \oplus E^{l,A}_{x}$ such that
\begin{displaymath}
A(x)E^{i,A}_{x}=E^{i,A}_{f(x)}  \textrm{  and  }
\lambda _i(A, x) =\lim _{n\rightarrow \infty} \dfrac{1}{n}\log \| A^n(x)v\| 
\end{displaymath}
for every non-zero $v\in E^{i,A}_{x}$ and $1\leq i \leq l$. Moreover, since $\mu$ is ergodic, the \textit{Lyapunov exponents} $\lambda _i(A, x)$ are constant on a full $\mu$-measure subset of $M$ (and thus we denote it just by $\lambda_i(A,\mu)$) as well as the dimensions of the \textit{Oseledets subspaces} $E^{i,A}_{x}$.
The dimension of $E^{i,A}_{x}$ is called the \emph{multiplicity} of $\lambda_i(A,\mu)$. When there is no risk of ambiguity, we suppress the index $A$ or even both $A$ and $\mu$ from the previous objects.

\subsection{Lyapunov norm} \label{sec: lyap norm} Let $x\in \Reg$ and $\varepsilon >0$. Given two vectors $u=u_1+\ldots +u_l $ and $v=v_1+\ldots +v_l $ in $\R ^d$ where $u_i, v_i \in E^{i,A}_x$ for every $1\leq i\leq l$, the \emph{$\varepsilon$-Lyapunov inner product} of $u$ and $v$ at $x$ is defined by 
\begin{displaymath}
\langle u,v\rangle _{x,\varepsilon}=d\sum _{i=1}^l\langle u_i,v_i \rangle _{x,\varepsilon,i}
\end{displaymath}
where 
\begin{equation}\label{eq: def Lyap i norm}
\langle u_i, v_i\rangle_{x,\varepsilon,i}= \sum _{n\in \mathbb{Z}}\langle A^{n}_i(x)u_i, A^{n}_i(x)v_i \rangle e^{-2\lambda _i n -2\varepsilon \mid n\mid}
\end{equation}
for every $i=1,\ldots ,l $. It follows from the Oseledets' theorem that this last series converge. We then define the \emph{$\varepsilon$-Lyapunov norm} $\norm{.}_{x,\varepsilon}$ associated to the cocycle $A$ at $x\in \Reg$ as the norm generated by $\langle \cdot ,\cdot \rangle _{x,\varepsilon}$. When there is no risk of ambiguity, we write $\norm{.}_x$ instead of $\norm{.}_{x,\varepsilon}$ and call it just \emph{Lyapunov norm}.

Given a liner map $B\in \GL$, its Lyapunov norm is defined for any regular points $x,y\in \Reg$ by
\begin{displaymath}
\norm{B}_{y\leftarrow x}=\sup \{\norm{Bu}_{y}/\norm{u}_{x}; \; u\in \R^d\setminus \{0\} \}.
\end{displaymath}  

Lyapunov norms are fundamental tools in the study of the asymptotic growths of a cocycle. The main properties of these objects that we are going to use in the sequel are the following (see \cite[Sections 3.5.1 to 3.5.3]{BP07} for a detailed discussion)
\begin{itemize}
\item For every $1\leq i\leq l$ and $u\in E^{i,A}_x$ we have
\begin{equation} \label{ineq: Lyapunov norm}
e^{(\lambda _i -\varepsilon)\mid n\mid}\norm{u}_{x}\leq \norm{A^n(x)u}_{f^n(x)}\leq e^{(\lambda _i +\varepsilon)\mid n\mid}\norm{u}_{x} 
\end{equation}
for every $n\in \Z$; 
\item There exists a measurable function $C_{\varepsilon}:\Reg \to (0,+\infty)$ such that
\begin{equation}\label{ineq: norm x Lyapunov norm}
\norm{u}\leq \norm{u}_x\leq C_{\varepsilon}(x)\norm{u}
\end{equation}
whose growth along any regular orbit is bounded; more precisely,
\begin{equation}\label{ineq: K delta growth}
C_{\varepsilon}(x)e^{-\varepsilon \mid n\mid}\leq C_{\varepsilon}(f^n(x))\leq C_{\varepsilon}(x)e^{\varepsilon \mid n \mid } \quad \forall n\in \Z;
\end{equation} 
\item In particular, 
\begin{equation} \label{ineq: norm growth}
\norm{A^n(x)u} \leq C_\varepsilon(x) e^{(\lambda _1 +\varepsilon)\mid n\mid}\norm{u} 
\end{equation}
for every $u\in \R^d$ and $n\in \Z$ and for any linear map $B$ and regular points $x$ and $y$,
\begin{equation}\label{ineq: norm x Lyapunov norm operator}
C_\varepsilon(x)^{-1}\norm{B}\leq \norm{B}_{y\leftarrow x}\leq C_\varepsilon(y)\norm{B}.
\end{equation}

\end{itemize}

For $N>0$, let $\Reg _{\varepsilon ,N}$ be the set of regular points $x\in \Reg$ for which $C_{\varepsilon}(x)\leq N$. Observe that $\mu(\Reg _{\varepsilon, N})\to 1$ as $N\to +\infty$. Moreover, invoking Lusin's theorem we may assume without loss of generality that this set is compact and that the Lyapunov norm and the Oseledets splitting are continuous when restricted to it.

\subsection{Non-uniformly hyperbolic systems}\label{sec: non unif hyp}

An $f$-invariant measure $\mu$ is said to be \textit{hyperbolic} if all Lyapunov exponents $\{\lambda _i(Df, \mu) \}_{i=1}^{l}$ are non-zero. When this happens, $(f,\mu)$ is called \textit{non-uniformly hyperbolic}. Given $\chi>0$, $\mu$ is called \textit{$\chi$-hyperbolic} if $0< \chi < \min \{ | \lambda _i(Df, \mu) | \colon 1\leq i\leq l\}$. Non-uniform hyperbolicity implies the existence of a very rich geometric structure of the dynamics of $f$, given by stable and unstable manifolds in the sense of Pesin (see Section \ref{sec: more non-unif hyp} where such properties are recalled). We now recall a result due to Katok \cite{Kat80} (see also Theorem 15.1.2 of \cite{BP07}) which describes one of the main properties of $f$ given by this geometric structure. In order to do so, let us denote by $\Lam_{\varepsilon,N}$ the set $\Reg_{\varepsilon,N}$ constructued in Section \ref{sec: lyap norm} associated to the cocycle generated by $A(x)=Df(x)$ over $f$. This set is usually called the \emph{Pesin set}.

\begin{theorem}[Closing Lemma for non-uniformly hyperbolic systems]\label{theo: closing lemma} 
Assume $\mu$ is $\chi$-hyperbolic and let $\varepsilon>0$ be sufficiently small when compared to $\chi$. Then, there are constants $C=C(\varepsilon,N) >0$ and $\eta=\eta(\varepsilon,N)\in (\varepsilon, \chi-\varepsilon)$ and for each $h>0$ there exists $\beta=\beta(h,\varepsilon,N)\in (0,h)$ so that if $y\in\Lam_{\varepsilon,N}$ satisfies $d(f^n(y),y)<\beta$ and $f^n(y)\in\Lam_{\varepsilon,N}$ then there exists a periodic point $p$ such that $f^n(p)=p$ and
\begin{displaymath}
d(f^i(y),f^i(p))\leq hC e^{-\eta \min\lbrace i, n-i\rbrace}
\end{displaymath}
for every $i=0,1,\ldots ,n$.
\end{theorem}

\section{Proof of Theorem \ref{theo: 1} }\label{sec:theo1}
Let $f:M\to M$ be a $C^{1+\epsilon}$ diffeomorphism, $\mu$ a $\chi$-hyperbolic ergodic $f$-invariant measure and $A:M\to \GL$ an $\alpha$-H\"{o}lder continuous map satisfying \eqref{eq: POOC}. Fix $\varepsilon'>0$ sufficiently small when compared to $\chi$ and $N'\in \N$ sufficiently large so that the Pesin set $\Lam_{\varepsilon',N'}$ has $\mu$-measure larger than $0.99$. Let $C>0$ and $\eta>0$ be given by the Closing Lemma \ref{theo: closing lemma} associated to these parameters. Fix $0<\varepsilon_0<\frac{1}{10}\min \{\alpha\eta,\varepsilon' \}$  and take $\varepsilon \in (0,\varepsilon_0)$ and $N\in \N$ sufficiently large so that the set $G=\Lam_{\varepsilon',N'}\cap \Reg_{\varepsilon,N}$ has $\mu$-measure larger than $0.9$.

We claim now that there exists $x\in \supp(\mu)$ so that $x\in G$, $\supp(\mu)\subset \overline{\O}(x)$ and $\O(x)\cap G$ is dense in $G$ where $\O(x)$ denotes the orbit of $x$ under $f$. Indeed, it is a classical fact that, in our setting, the set of points whose orbit is dense in $\supp(\mu)$ has full $\mu$-measure. Similarly for the first return map to $G$ and the restriction of $\mu$ to $G$. Thus, combining these facts we can get a point $x\in G$ with the desired properties. 

Fix $x\in M$ as above. We then define $P:\O(x) \to \GL $ as
\begin{displaymath}
P(f^n(x))=A^n(x) \text{ for every } n\in \N.
\end{displaymath}
It is easy to see that $P$ satisfies $A(z)=P(f(z))P(z)^{-1}$ for every $z\in \O(x)$. Our objective now is to prove that $P$ is uniformly continuous when restricted to $\O(x)\cap G$ so that it can be continuously extended to $\overline{\O(x)\cap G}=G$. Denote by $\overline{P}:G\to \GL$ such extesion restricted to $G$. Then, for $z\in f(G)\setminus G$ we define 
$$\overline{P}(z)=A(f^{-1}(z))\overline{P}(f^{-1}(z)).$$
Recursively, we define $\overline{P}$ via the same expression for $z\in \cup_{j=0}^{n}f^j(G)\setminus \cup_{j=0}^{n-1}f^j(G)$ for every $n\in \N$. Consequently, since $\mu(\cup_{j\in \N}f^j(G))=1$, we get a map $\overline{P}$ defined almost everywhere and satisfying 
$$A(x)=\overline{P}(f(x))\overline{P}(x)^{-1} \text{ for $\mu$-almost every } x\in M$$
as we claimed. So, all we have to do is to prove that $P$ is uniformly continuous when restricted to $\O(x)\cap G$. In order to do that, we are going to prove that there exists a constant $K>0$ so that for every $h\in (0,1)$ there exists $\beta=\beta(h) >0$ so that if $z,y\in \O(x)\cap G$ satisfy $d(y,z)<\beta$ then
\begin{equation}\label{eq: unif cont}
\norm{P(z)-P(y)}\leq Kh^\alpha.
\end{equation}

The next proposition is the main step in proving \eqref{eq: unif cont}. In order to prove it, we need the following simple observation: from Theorem 1.4 of \cite{KS2} and \eqref{eq: POOC} it follows that $\lambda_1(A,\mu)=0=\lambda_l(A,\mu)$. In particular, it follows from \eqref{ineq: norm growth} and the definition of $\Reg_{\varepsilon,N}$ that
\begin{equation}\label{ineq: norm growth G}
\norm{A^n(z)} \leq N e^{\varepsilon \mid n\mid} 
\end{equation}
for every $z\in G$ and $n\in \Z$.

\begin{proposition}\label{prop: unif cont points close}
Given $h\in (0,1)$, let $\beta>0$ be given by the Closing Lemma \ref{theo: closing lemma} associated to $C, \eta$ and $h$. Then, there exists $\tilde{C}>0$ independent of $h$ so that if $z,f^n(z)\in \O(x)\cap G$ satisfy $d(z,f^n(z))<\beta$ then
\begin{displaymath}
\norm{A^n(z)-\Id}\leq \tilde{C}h^\alpha.
\end{displaymath}
\end{proposition}

\begin{proof} We are going to consider the case when $n=2m$ for some $m\in \N$. The case when $n$ is odd is similiar. By the Closing Lemma \ref{theo: closing lemma} there exists a periodic point $p\in M$ such that $f^{2m}(p)=p$ and
\begin{displaymath}
d(f^i(z),f^i(p))\leq hC e^{-\eta \min\lbrace i, 2m-i\rbrace}
\end{displaymath}
for every $i=0,1,\ldots ,2m$. In particular, 
\begin{displaymath}
d(f^i(z),f^i(p))\leq hC e^{-\eta i} \text{ and } d(f^{2m-i}(z),f^{2m-i}(p))\leq hC e^{-\eta i}
\end{displaymath}
for every $i=0,1,\ldots ,m$. We will need the next auxiliary result.

\begin{lemma}\label{lem: domination}
There exists a constant $L>0$ independent of $z,p$ and $m$ so that 
\begin{displaymath}
\norm{A^i(p)^{-1}}\leq L e^{2\varepsilon i} \text{ and }\norm{A^i(f^{2m-i} (p))^{-1}}\leq L e^{2\varepsilon i}
\end{displaymath}
for every $i=0,1,\ldots ,m$.
\end{lemma}
\begin{proof}[Proof of Lemma \ref{lem: domination}] We prove only the first inequality. The second one is similar. Fix $i\in \{0,1,\ldots ,m\}$ and for each $j\in \{0,1,\ldots,i\}$ let us consider $B_j=A(f^j(p))^{-1}-A(f^j(z))^{-1}$. Our first objective is to estimate $\norm{A^i(p)^{-1}}_{z\leftarrow f^i(z)}$. We start observing that
\begin{displaymath}
\begin{split}
\norm{A^i(p)^{-1}}_{z\leftarrow f^i(z)}&=\norm{A(p)^{-1}\cdot \ldots \cdot A(f^{i-1}(p))^{-1}}_{z\leftarrow f^i(z)}\\
&= \norm{ (A(z)^{-1}+B_0) \cdot \ldots \cdot  (A(f^{i-1}(z))^{-1}+B_{i-1}) }_{z\leftarrow f^i(z)}\\
&\leq \norm{A(z)^{-1}+B_{0} }_{z\leftarrow f(z)}\cdot \ldots \cdot \norm{A(f^{i-1}(z))^{-1}+B_{i-1}  }_{f^{i-1}(z)\leftarrow f^i(z)}.
\end{split}
\end{displaymath}

Now, since $A$ is $\alpha$-H\"{o}lder, there exists a constant $L_1=L_1(A)>0$ so that
\begin{displaymath}
\begin{split}
\norm{B_j}& \leq L_1 d(f^j(z),f^j(p))^\alpha\leq L_1 h^\alpha C^\alpha e^{-\eta \alpha j}
\end{split}
\end{displaymath}
which combined with \eqref{ineq: norm x Lyapunov norm operator} and \eqref{ineq: K delta growth} gives us
\begin{equation*}
\begin{split}
\norm{B_j}_{f^{j}(z)\leftarrow f^{j+1}(z)} \leq C_\varepsilon(f^{j}(z))\norm{B_j} \leq NL_1 h^\alpha C^\alpha e^{(\varepsilon-\eta \alpha) j}.
\end{split}
\end{equation*}
Moreover, recalling that $\lambda_1(A,\mu)=0=\lambda_l(A,\mu)$ and using \eqref{ineq: Lyapunov norm} we have that
\begin{displaymath}
\norm{A(f^{j}(z))^{-1}}_{f^{j}(z)\leftarrow f^{j+1}(z)} \leq e^{\varepsilon}.
\end{displaymath}
Thus, combining the previous two inequalities 
\begin{displaymath}
\begin{split}
\norm{A(f^{j}(z))^{-1}+B_{j}}_{f^{j}(z)\leftarrow f^{j+1}(z)}& \leq \norm{A(f^{j}(z))^{-1}}_{f^{j}(z)\leftarrow f^{j+1}(z)} +\norm{B_{j}}_{f^{j}(z)\leftarrow f^{j+1}(z)}\\
&\leq e^{\varepsilon } +N L_1h^\alpha C^\alpha e^{(\varepsilon-\eta \alpha) j}  \\
& = e^{\varepsilon} (1+ e^{-\varepsilon}NL_1 h^\alpha C^\alpha e^{(\varepsilon-\eta \alpha) j}).
\end{split}
\end{displaymath}
Making $L_2= e^{-\varepsilon}NL_1 h^\alpha C^\alpha$ and using the fact that $1+y\leq e^y$ for every $y\geq 0$ it follows that
\begin{displaymath}
\norm{A(f^{j}(z))^{-1}+B_{j}}_{f^{j}(z)\leftarrow f^{j+1}(z)} \leq e^{\varepsilon} \exp (L_2 e^{(\varepsilon-\eta \alpha) j}).
\end{displaymath}
Consequently,
\begin{displaymath}
\begin{split}
\norm{A^i(p)^{-1}}_{z\leftarrow f^i(z)}& \leq \prod _{j=0} ^{i-1}  e^{\varepsilon} \exp (L_2 e^{(\varepsilon-\eta \alpha) j}) \\
&= e^{\varepsilon i}   \exp \left( L_2 \sum _{j=0}^{i-1} e^{(\varepsilon-\eta \alpha) j} \right).
\end{split}
\end{displaymath}
Thus, recalling that $\varepsilon-\eta \alpha <0$ and making $L_3=\exp \left( L_2 \sum _{j=0}^{\infty} e^{(\varepsilon-\eta \alpha) j}) \right)$ we get that 
\begin{equation*}
\norm{A^i(p)^{-1}}_{z\leftarrow f^i(z)} \leq L_3 e^{\varepsilon i}
\end{equation*}
for every $i\in \{0,1,\ldots,m\}$. Now, from \eqref{ineq: norm x Lyapunov norm operator} and \eqref{ineq: K delta growth} it follows that
\begin{displaymath}
\norm{A^i(p)^{-1}}\leq C_\varepsilon(f^i(z))\norm{A^i(p)^{-1}}_{z\leftarrow f^i(z)} \leq C_{\varepsilon}(z)e^{i\varepsilon} \norm{A^i(p)^{-1}}_{z\leftarrow f^i(z)}.
\end{displaymath}
Thus, taking $L=L_3N$ and recalling that $C_\varepsilon(z)\leq N$ (once $z\in G$) it follows that
\begin{displaymath}
\norm{A^i(p)^{-1}}\leq L e^{2\varepsilon i}
\end{displaymath}
as claimed.
\end{proof}

Going back to the proof of Proposition \ref{prop: unif cont points close}, we start observing that 
\begin{equation}\label{eq: auxil 1}
\norm{A^m(p)^{-1}A^m(z)-\Id}\leq \tilde{C}_1h^\alpha
\end{equation}
for some $\tilde{C_1}>0$. In fact,
\begin{displaymath}
\norm{A^m(p)^{-1}A^m(z)-\Id}
\end{displaymath}
is smaller than or equal to
\begin{equation*}
 \sum_{j=0}^{m-1} \norm{A^{j}(p)^{-1}A^{j}(z)-A^{j+1}(p)^{-1}A^{j+1}(z)}.
\end{equation*}
By the cocycle property the previous expression is equal to
\begin{equation*}
\sum_{j=0}^{m-1} \norm{A^{j}(p)^{-1}\left(\Id-A(f^j(p))^{-1}A(f^j(z))\right) A^{j}(z)}  \end{equation*}
which by the property of the norm is smaller than or equal to
\begin{displaymath}
\sum ^{m-1}_{j=0} \norm{A^{j}(p)^{-1}}\norm{A^{j}(z)} \norm{\Id-A(f^j(p))^{-1}A(f^j(z))}.
\end{displaymath}
Now, since $A$ is $\alpha$-H\"{o}lder continuous, there exist a constant $\hat{C}=\hat{C}(A)>0$ such that 
$$\norm{\Id-A(f^j(p))^{-1}A(f^j(z))} \leq \hat{C}d(f^j(z),f^j(p))^\alpha \leq \hat{C} h^\alpha C^\alpha e^{-\eta \alpha  j}$$
for every $j=0,1, \ldots ,m$. Plugging it to the previous expression and using Lemma \ref{lem: domination} and \eqref{ineq: norm growth G} it follows that 
\begin{displaymath} 
\norm{A^m(p)^{-1}A^m(z)-\Id}
\leq \sum ^{m}_{j=0}L e^{2\varepsilon j}N e^{\varepsilon j} \hat{C} h^\alpha C^\alpha e^{-\eta \alpha j}.
\end{displaymath}
Thus, taking $\tilde{C}_1= \sum ^{\infty}_{j=0}L N  \hat{C}  C^\alpha e^{(3\varepsilon -\eta \alpha ) j} $ (recall that $0<10\varepsilon <\eta \alpha$) the claim follows. In particular, for $h\in (0,1)$,
\begin{equation}\label{eq: auxil bound}
\norm{A^m(p)^{-1}A^m(z)}\leq \tilde{C}_1+1.
\end{equation}
Similarly, using the second inequality in Lemma \ref{lem: domination} and the fact that $f^{2m}(z)\in G$,
\begin{equation}\label{eq: auxil 2}
\norm{A^m(f^m(z))A^m(f^m(p))^{-1}-\Id} \leq \tilde{C}_2 h^\alpha.
\end{equation}
Now, since $A^{2m}(p)=\Id$ we get that $A^m(f^m(p))^{-1}A^m(p)^{-1}=\Id$. Consequently, combining it with \eqref{eq: auxil 1}, \eqref{eq: auxil bound} and \eqref{eq: auxil 2}  we get that
\begin{displaymath}
\begin{split}
\norm{A^{2m}(z)-\Id}&= \norm{A^m(f^m(z))A^m(f^m(p))^{-1} A^m(p)^{-1}A^m(z)-\Id}\\
&\leq \norm{A^m(f^m(z))A^m(f^m(p))^{-1} A^m(p)^{-1}A^m(z)-A^m(p)^{-1}A^m(z) } \\
& + \norm{ A^m(p)^{-1}A^m(z)-\Id }\\
&\leq \norm{A^m(p)^{-1}A^m(z)}\norm{A^m(f^m(z))A^m(f^m(p))^{-1} -\Id } \\
&+ \norm{ A^m(p)^{-1}A^m(z)-\Id }\\
&\leq (\tilde{C}_1+1)\tilde{C}_2h^\alpha +\tilde{C}_1h^\alpha=\tilde{C}h^\alpha
\end{split}
\end{displaymath}
with $\tilde{C}=\tilde{C}_1+(\tilde{C}_1+1)\tilde{C}_2$ as claimed.
\end{proof}

\begin{corollary}\label{cor:bound T}
There exists a constant $T>0$ so that if $f^m(x)\in G$ then
\begin{displaymath}
\norm{A^m(x)}\leq T.
\end{displaymath}
\end{corollary}

\begin{proof} Since $G$ is compact and $\mathcal{O}(x)\cap G$ is dense in it there exists a subsequence $i_1<\ldots<i_N $ so that $\{B(f^{i_j}(x),\frac{\beta}{2})\}_{j=0}^{N}$ is an open cover of $G$ and $f^{i_j}(x)\in G$ for every $j=0,1,\ldots,N$. Let $T_1=\max_{t=0,1,2,\ldots,i_N}\{\max_{y\in M}\norm{A^t(y)}\}$ and $T=T_1(\tilde{C}+1)$. Thus, if $m\leq i_N$ then $\norm{A^m(x)}\leq T_1\leq T$ as claimed. Suppose $m>i_N$. Since $\{B(f^{i_j}(x),\frac{\beta}{2})\}_{j=0}^{N}$ is an open cover of $G$ and $f^{m}(x)\in G$, there exists $j\in \{0,1,\ldots,N\}$ so that $d(f^{i_j}(x),f^{m}(x))<\beta$. In particular, by Proposition \ref{prop: unif cont points close},
$$\norm{A^{m-i_j}(f^{i_j}(x))}\leq \tilde{C}+1.$$
Consequently,
\begin{displaymath}
\begin{split}
\norm{A^m(x)}&\leq \norm{A^{m-{i_j}}(f^{i_j}(x))A^{i_j}(x)} \\
&\leq \norm{A^{m-{i_j}}(f^{i_j}(x))}\norm{A^{i_j}(x)}\leq T
\end{split}
\end{displaymath} 
as claimed.
\end{proof}

Let $z,y\in \O(x)\cap G$ be such that $d(y,z)<\beta$. In particular, there exist $m,n\in \N$ so that $z=f^m(x)$ and $y=f^n(x)$ and we may assume without loss of generality that $n>m$. Thus, using Corollary~\ref{cor:bound T} and Proposition \ref{prop: unif cont points close},
\begin{displaymath}
\begin{split}
\norm{P(y)-P(z)}&=\norm{P(f^n(x))-P(f^m(x))}=\norm{A^n(x)-A^m(x)}\\
&=\norm{A^m(x)}\norm{A^{n-m}(f^m(x))-\Id} \leq T\tilde{C}h^\alpha =Kh^\alpha
\end{split}
\end{displaymath}
where $K=T\tilde{C}$, proving \eqref{eq: unif cont} and concluding the proof of Theorem \ref{theo: 1}. \qed

\section{H\"older continuity on large sets} \label{sec: holder reg}

In this section we prove Theorem \ref{theo: regularity}. This is a counterpart to the results of \cite{Pol05} where the case of cocycles taking values in $(\R,+)$ or any compact group was considered. Let $f:M\to M$ be a $C^{1+\epsilon}$ diffeomorphism, $\mu$ a hyperbolic ergodic $f$-invariant measure and $A:M \to \GL $ an $\alpha$-H\"{o}lder continuous map. We start by recalling some useful constructions and results from \cite{almost}.

\subsection{More on non-uniformly hyperbolic systems}\label{sec: more non-unif hyp} By Pesin's stable manifold theorem (see \cite{BP07}), there exists a full $\mu$-measure set $H(\mu) \subset M$ so that through every point $x\in H(\mu)$ there exist $C^1$ embedded disks $W^s_{\text{\rm loc}}(x)$ and $W^u_{\text{\rm loc}}(x)$, called \textit{local stable and unstable sets} at $x$, such that
\begin{itemize}
\item[i)] $W^s_{\text{\rm loc}}(x)$ is tangent to $E^s_x$ and $W^u_{\text{\rm loc}}(x)$ is tangent to $E^{u}_x$ where
$$E^s_x=\bigoplus_{\lambda _i(Df, \mu)<0}E^{i,Df}_x  \text{ and } E^u_x=\bigoplus_{\lambda _i(Df, \mu)>0}E^{i,Df}_x; $$
\item[ii)] given $0<\tau _x <\min _{1\leq i\leq k}| \lambda _i(Df,\mu)|$ there exists $C_x>0$ such that 
$$
\left\{
\begin{array}{rl}
d(f^n(y),f^n(z))\leq C_xe^{-\tau _x n}d(y,z) & ,\forall y,z\in W^s_{\text{\rm loc}}(x),\forall n\geq 0,\\
d(f^{-n}(y),f^{-n}(z))\leq C_xe^{-\tau _x n}d(y,z) & ,\forall y,z\in W^u_{\text{\rm loc}}(x),\forall n\geq 0;
\end{array}
\right.
$$
\item[iii)] $f(W^s_{\text{\rm loc}}(x))\subset W^s_{\text{\rm loc}}(f(x))$ and $f(W^u_{\text{\rm loc}}(x)) \supset W^u_{\text{\rm loc}}(f(x))$;
\item[iv)] $W^s(x)=\bigcup _{n=0}^{\infty} f^{-n}(W^s_{\text{\rm loc}}(f^n(x)))$ and $W^u(x)=\bigcup _{n=0}^{\infty} f^{n}(W^u_{\text{\rm loc}}(f^{-n}(x)))$.
\end{itemize}
Moreover, $W^s_{\text{\rm loc}}(x)$ and $W^u_{\text{\rm loc}}(x)$ depend measurably on $x$, as $C^1$ embedded disks, as well as the constants $\tau _x$ and $C_x$. By Lusin's theorem, we may find compact \textit{hyperbolic blocks} $\mathcal{H}(C,\tau)$, whose measure can be made arbitrarily close to 1 by increasing $C$ and decreasing $\tau$, such that in $\mathcal{H}(C,\tau)$ the sets $W^s_{\text{\rm loc}}(x)$ and $W^u_{\text{\rm loc}}(x)$ vary continuously,
$\tau _x>\tau \text{ and } C_x<C$.
In particular, the sizes of $W^s_{\text{\rm loc}}(x)$ and $W^u_{\text{\rm loc}}(x)$ are uniformly bounded from zero on $\mathcal{H}(C,\tau)$, as well as the angles between these disks.
The drawback on this argumentation is that  $\mathcal{H}(C,\tau)$ is in general not $f$-invariant. 

Given $x\in \cHCT$ and $\delta=\delta(C,\tau)>0$ small enough, for every $y\in \overline{B(x,\delta)}\cap \cHCT$, $W^u_{\text{loc}}(y)$ intersects $W^s_{\text{loc}}(x)$ at exactly one point and similarly $W^u_{\text{loc}}(x)$ intersects $W^s_{\text{loc}}(y)$ at exactly one point. Let
 \begin{displaymath}
 \mathcal{N}^u_x(\delta)\subset W^u_{\text{loc}}(x) \text{ and } \mathcal{N}^s_x(\delta)\subset W^s_{\text{loc}}(x) 
 \end{displaymath}
 be the compact sets of all intersection points obtained by the previous procedure when $y$ varies in $\mathcal{H}(C,\tau)\cap \overline{B(x,\delta)}$. Reducing $\delta >0$ if necessary, $W^s_{\text{loc}}(z)\cap W^u_{\text{loc}}(w)$ consists of exactly one point $[z,w]$ for every $z\in \mathcal{N}^u_x(\delta)$ and $w\in \mathcal{N}^s_x(\delta)$. Let $\mathcal{N}_x(\delta)$ be the image of $\mathcal{N}^u_x(\delta)\times \mathcal{N}^s_x(\delta)$ under the map
 \begin{equation}\label{eq: f local product struture}
 (z,w)\to [z,w].
 \end{equation}
 By construction, $ \mathcal{N}_x(\delta)$ contains $\overline{B(x,\delta)}\cap \mathcal{H}(C,\tau)$ and is homeomorphic, via \eqref{eq: f local product struture},
 to $\mathcal{N}^u_x(\delta)\times \mathcal{N}^s_x(\delta)$.

We say that $\mu$ has \emph{local product structure} if for every point $x\in \text{supp}(\mu)$ and every small $\delta >0$, the measure $\mu_x=\mu\mid \cN_x(\delta)$ is equivalent to the product $\mu^u_x\times\mu_x^s$ where $\mu_x^u$ and $\mu_x^s$ are the projections of $\mu_x $ to $\mathcal{N}^u_x(\delta)$ and $\mathcal{N}^s_x(\delta)$, respectively.

\subsection{Stable and Unstable Holonomies} Suppose $\mu$ has local product structure. Let $\cHCT$ be a hyperbolic block with constants $C>0$ and $\tau >0$ as in the previous section. Given $N\geq 1$ and $\theta>0$, let $\cDN$ be the set of points $x\in M$ satisfying
\begin{equation}
\prod_{j=0}^{k-1}\norm{A^N(f^{jN}(x)}\norm{A^N(f^{jN}(x)^{-1}}\leq e^{\theta kN},\quad\text{ for all } k\geq1,
\end{equation} 
and the dual condition which is obtained replacing $f$ and $A$ by its inverses.

A key notion that we are going to use is that of \emph{s-dominitation}: given $s\geq 1$, we say that $x\in M$ is $s$-\emph{dominated for} $A$ if $x\in\cHCT\cap \cDN$ for some $C,\tau,N,\theta$ with $s\theta<\tau$.

As a first result we have that almost every point with small Lyapunov exponents is dominated. More precisely,

\begin{lemma}[Corollary 2.4 of \cite{almost}] \label{c.dominated}
Given $\theta>0$ and $\lambda \geq 0$ satisfying $\theta>d\lambda\geq 0$, then $\mu$-almost every $x\in M$ with $\lambda_1(A,x)\leq \lambda$ is in $\cDN$ for some $N\geq 1$. In particular, $\mu$-almost every $x\in M$ with $\lambda_1(A,x)=0$ is $s$-dominated for $A$, for every $s\geq 1$.
\end{lemma}

As an important consequence of domination we have the existence of \emph{stable and unstable holonomies}. This objects play a major part in our argument as well as in many results about fiber-bunched cocycles. Their existence and main properties are given in the proposition below which is a combination of Proposition 2.5, Corollary 2.8 and Lemma 2.9 of \cite{almost}.

\begin{proposition}\label{p.holonomy}
Given constants $C,\tau,N$ and $\theta$ bigger than zero with $2\theta<\tau$, there exists $L>0$ such that for any $x\in \cHCT\cap\cDN$ and $y,z\in \Ws(x)$ the limit
$$H^{s,A}_{f^j(y)f^j(z)}=\lim_{n\to+\infty}{A^n(f^j(z))}^{-1}A^n(f^j(y))$$
exists for every $j\geq 0$, and satisfy $H^{s,A}_{f^j(y)f^j(z)}=A^j(z)H^{s,A}_{yz}A^j(y)^{-1}$ and
$$
\norm{H^{s,A}_{f^j(y)f^j(z)}-\Id}\leq L \d(y,z)^\alpha.
$$
\end{proposition}
The family of maps $H^{s,A}_{yz}$ is called \emph{stable holonomies} for the cocycle $(f,A)$. Similarly, for $y,z\in \Wu(x)$ we have the family of \emph{unstable holonomies} given by 
$$H^{u,A}_{yz}=\lim_{n\to+\infty}{A^{-n}(z)}^{-1}A^{-n}(y).$$

We call \emph{holonomy block for $A$} any compact set $\cO$ satisfying $\cO\subset \cHCT\cap \cDN$ for some $C,\tau,N,\theta$ with $2\theta<\tau$.

\subsection{Proof of Theorem \ref{theo: regularity}}  Suppose the measure $\mu$ has local product structure and that there exists a measurable map $P:M\to GL(d,\R)$ satisfying \eqref{eq: trasnf map theo}. Moreover, assume that $\mu$ is non-atomic, otherwise the theorem is trivial. As one can easily verify, identity \eqref{eq: trasnf map theo} implies that all the Lyapunov exponents of $(f,A)$ with respect to $\mu$ are equal to zero. Consequently, Lemma \ref{c.dominated} implies that $\mu$-almost every $x\in M$ is $s$-dominated for $A$, for every $s\geq 1$. In particular, for every $\varepsilon >0$ we can take an holonomy block $\cO\subset \cHCT\cap \cDN$ for $A$ for some $C,\tau,N,\theta$ satisfying $\mu(\cO)>1-\frac{\varepsilon}{2}$. Moreover, since $P$ is measurable, by Lusin's theorem there exists a compact set $X_\varepsilon \subset M$ with $\mu(X_\varepsilon)>1-\frac{\varepsilon}{2}$ so that $P$ restricted to it is uniformly continuous. We assume without loss of generality that $\varepsilon < 0.1$.

In order to conclude our proof, we need the following two auxiliary results which are versions of Lemma 2.2 and 2.3 of \cite{Pol05}, respectively, adapted to our setting. 

\begin{lemma}\label{l.abscont}
There exists a set $\tilde{X}_{\varepsilon}\subset X_\varepsilon \cap \cO$ with $\mu(\tilde{X}_{\varepsilon})=\mu(X_\varepsilon\cap \cO)$ so that 
\begin{equation}\label{eq.recur}
\lim_{n\to \pm\infty} \frac{1}{n}\# \lbrace 1\leq i\leq n; \; f^i(x)\in X_\varepsilon\cap \cO\rbrace>\frac{1}{2}
\end{equation}
for every $x\in \tilde{X}_{\varepsilon}$. Moreover, for $\mu_x^u$-almost every $y\in \Wu(x)$ and $\mu_x^s$-almost every $z\in \Ws(x)$ equation \eqref{eq.recur} is satisfied where $\mu_x^u$ and $\mu_x^s$ are the induced measures as in the definition of local product structure.
\end{lemma}
\begin{proof}
By Birkhoff's ergodic theorem there exists a full $\mu$-measure subset $X_1\subset M$ so that equation \eqref{eq.recur} is satisfied for every $x\in X_1$. On the other hand, from the definition of Rokhlin disintegration, there exists a full $\mu$-measure set $X_2\subset M$ so that for every $x\in X_2$, $\mu^s_x(\Ws(x)\setminus X_1)=0$. Similarly, there exists a full $\mu$-measure set $X_3\subset M$ so that for every $x\in X_3$, $\mu^u_x(\Wu(x)\setminus X_1)=0$. Thus, taking $\tilde{X}_{\varepsilon}=X_\varepsilon \cap \cO \cap X_1\cap X_2\cap X_3$ the proof is complete. 
\end{proof}

Let us consider $\Delta_\varepsilon=\tilde{X}_{\varepsilon}\cap \supp(\mu)$. In particular, $\mu(\Delta_\varepsilon)=\mu(\tilde{X}_{\varepsilon})=\mu(X_\varepsilon\cap \cO)>1-\varepsilon$.

\begin{lemma}\label{l.3point}
Let $\delta>0$ be small enough as in the definition of local product structure. Thus,
\begin{itemize}
\item for $\mu$-almost every $x,y \in \Delta_\varepsilon$ with $\d(x,y)<\delta/2$, one can find $x_1,x_2$ and $x_3$ such that $x_1\in \Wu(x)$, $x_2\in \Ws(y)$, equation \eqref{eq.recur} is true for $x_1$ and $x_2$, and $x_3\in \Wu(x_2)\cap \Ws (x_1)\cap \Delta_\varepsilon$;
\item there exists $K>0$ such that $K\d(x,y)\geq \d(x,x_1)+\d(x_1,x_3)+\d(x_3,x_2)+\d(x_2,y)$.
\end{itemize}
\end{lemma}
\begin{proof}
Let $x,y\in \Delta_\varepsilon$ be so that $\d(x,y)<\delta/2$. By the local product structure there exists a point $z\in \cHCT$ so that $x,y\in \mathcal{N}_z(\delta)$ and $\mu_z=\mu\mid \mathcal{N}_z(\delta) \sim \mu^u_z\times\mu_z^s$. Moreover, we may assume that $x\in \mathcal{N}^u_z(\delta)$ and $y\in \cN_z^s(\delta)$.

We start observing that, since $\mu$ is non-atomic and $x,y\in \Delta_\varepsilon$, $(\mu_z^u\times \mu_z^s)(\Delta_\varepsilon)>0$. Thus, by Fubini's theorem there exists $x_2\in \cN^s_z(\delta)$ so that equation \eqref{eq.recur} is true for it and moreover, $\mu_z^u(\Wu(x_2)\cap \Delta_\varepsilon)>0$. Consider the set $W=\{[w,x]; w\in\Wu(x_2)\cap \Delta_\varepsilon \} $. Recalling that $\mu^u_z$ is the projection of $\mu\mid \cN_z(\delta)$ by the stable holonomies, it follows that $\mu^u_z(W)=\mu_z^u(\Wu(x_2)\cap \Delta_\varepsilon)>0$. Now, from Lemma~\ref{l.abscont} we know that $\mu_z^u$-almost every point in $W$ satisfies \eqref{eq.recur}. Let $x_1\in W$ be any such point. Thus, since by construction $x_3=[x_1,x_2]\in\Delta_\varepsilon$ we conclude the proof of the first part of the lemma. 

The second part follows from the continuity of the local stable and unstable manifolds when restricted to $\Delta_\varepsilon$.
\end{proof}

We are now in position to conclude the proof of Theorem \ref{theo: regularity}. More precisely, we are going to prove that $P$ is $\alpha$-H\"{o}lder continuous when restricted to $\Delta_\varepsilon$. In order to do it, let $x,y\in \Delta_\varepsilon$ be so that $\d(x,y)<\delta/2$ and $x_1,x_2$ and $x_3$ be given by Lemma~\ref{l.3point} associated to $x$ and $y$. Fix $T>0$ so that $\norm{P\mid X_\varepsilon }\leq {T}$ and $\norm{P^{-1}\mid X_\varepsilon }\leq {T}$ (recall the choice of $X_\varepsilon$).
We start observing that
\begin{displaymath}
\begin{split}
\norm{P(y)-P(x_2)}&=\norm{(\Id-P(x_2)P(y)^{-1})P(y)} \\ 
&\leq \norm{\Id-P(x_2)P(y)^{-1}}\norm{P(y)} \leq T\norm{\Id-P(x_2)P(y)^{-1}}.
\end{split}
\end{displaymath}
Now, from the choice of $x_2$ and Lemma \ref{l.abscont} we get that there exists a subsequence $\{m_k\}_{k\in \N}$ satisfying $m_k\to +\infty$ so that for every $k\in \N$, $f^{m_k}(y),f^{m_k}(x_2)\in X_\varepsilon\cap \cO$. Thus, restricting ourselves to this subsequence, 
\begin{displaymath}
\begin{split}
\norm{\Id-P(x_2)P(y)^{-1}}&= \norm{\Id -A^{m_k}(x_2)^{-1}P(f^{m_k}(x_2))P(f^{m_k}(y))^{-1}A^{m_k}(y)}\\
&\leq \norm{A^{m_k}(x_2)^{-1}P(f^{m_k}(x_2))P(f^{m_k}(y))^{-1}A^{m_k}(y)-A^{m_k}(x_2)^{-1} A^{m_k}(y)} \\
&+ \norm{A^{m_k}(x_2)^{-1} A^{m_k}(y)-\Id}\\
&\leq \norm{A^{m_k}(x_2)^{-1}} \norm{P(f^{m_k}(x_2))P(f^{m_k}(y))^{-1}-\Id}\norm{A^{m_k}(y)} \\
&+ \norm{A^{m_k}(x_2)^{-1} A^{m_k}(y)-\Id}\\
&\leq T^4 \norm{P(f^{m_k}(x_2))P(f^{m_k}(y))^{-1}-\Id} \\
&+ \norm{A^{m_k}(x_2)^{-1} A^{m_k}(y)-\Id}. 
\end{split}
\end{displaymath}
Thus, making $k\to +\infty$ and recalling that $x_2\in \Ws(y)$ it follows that
\begin{displaymath}
\begin{split}
\norm{\Id-P(x_2)P(y)^{-1}}&\leq \norm{H^{s,A}_{yx_2}-\Id}.
\end{split}
\end{displaymath}
Combining these observations with Proposition \ref{p.holonomy} we get that
\begin{displaymath}
\norm{P(y)-P(x_2)}\leq  TL\d(y,x_2)^\alpha.
\end{displaymath}

Analogously we have that $\norm{P(x_2)-P(x_3)}\leq T L\d(x_2,x_3)^\alpha$, $\norm{P(x_3)-P(x_1)}\leq T L\d(x_3,x_1)^\alpha$ and $\norm{P(x_1)-P(x)}\leq T L\d(x_1,x)^\alpha$. Thus, using Lemma~\ref{l.3point} we conclude that 
\begin{displaymath}
\norm{P(y)-P(x)}\leq C_\varepsilon\d(y,x)^\alpha
\end{displaymath}
for some $C_\varepsilon>0$ independent of $x$ and $y$ completing the proof of Theorem \ref{theo: regularity}. \qed


\medskip{\bf Acknowledgements.} The first author was partially supported by a CAPES-Brazil postdoctoral fellowship under Grant No. 88881.120218/2016-01 at the University of Chicago.
The second author was partially supported by Fondation Louis D-Institut de France (project coordinated by M. Viana)


\end{document}